\documentclass[12pt,a4paper, oneside]{amsart}

\usepackage[utf8]{inputenc}
\usepackage{amsmath}
\usepackage{amssymb}
\usepackage{amsthm}
\usepackage{graphicx}

\newtheorem{theorem}{Theorem}
\newtheorem{lemma}{Lemma}

\newtheorem{definition}{Definition}
\newtheorem{corollary}{Corollary}
\newtheorem{remark}{Remark}
\begin{document}
\author{Csaba Vincze, Márk Oláh and Let\'{i}cia Lengyel}
\footnotetext[2]{{\bf MR subject classification:} 51M04}
\footnotetext[3]{Cs. Vincze is supported by the EFOP-3.6.2-16-2017-00015 project. The project has been supported by the European Union, co-financed by the European Social Fund. Let\'{i}cia Lengyel is supported by the University of Debrecen (Summer Grant 2019).}
\address{Cs. Vincze\\
Institute of Mathematics, University of Debrecen\\ P.O.Box 400, H-4002 Debrecen, Hungary}
\email{csvincze@science.unideb.hu}
\address{M. Ol\'{a}h\\
Institute of Mathematics, University of Debrecen,  Doctoral School of Mathematical and Computational Sciences
\\ P.O.Box 400, H-4002 Debrecen, Hungary}
\email{olma4000@gmail.com}
\address{Let\'{i}cia Lengyel\\
BSC Mathematics, University of Debrecen, \\ P.O.Box 400, H-4002 Debrecen, Hungary}
\email{letikee123@gmail.com}
\title[On equidistant polytopes ...] {On equidistant polytopes in the Euclidean space}

\begin{abstract}
An equidistant polytope is a special equidistant set in the space $\mathbb{R}^n$ all of whose boundary points have equal distances from two finite systems of points. Since one of the finite systems of the given points is required to be in the interior of the convex hull of the other one we can speak about inner and outer focal points of the equidistant polytope. It is of type $(q, p)$, where $q$ is the number of the outer focal points and $p$ is the number of the inner focal points. The equidistancy is the generalization of convexity because a convex polytope can be given as an equidistant polytope of type $(q, 1)$, where $q\geq n+1$. In the paper we present some general results about the basic properties of the equidistant polytopes: convex components, graph representations, connectedness, correspondence to the Voronoi decomposition of the space etc. Especially, we are interested in equidistant polytopes of dimension $2$ (equidistant polygons). Equidistant polygons of type $(3,2)$ will be characterized in terms of a constructive (ruler-and-compass) process to recognize them. In general they are pentagons with exactly two concave angles such that the vertices, where the concave angles appear at, are joined by an inner diagonal related to the adjacent sides of the polygon in a special way via the three reflection theorem for concurrent lines. The last section is devoted to some special arrangements of the focal points to get the concave quadrangles as equidistant polygons of type $(3,2)$.
\end{abstract}

\maketitle
\section{Introduction}

Let $K\subset \mathbb{R}^n$ be a subset in the Euclidean coordinate space. The distance between a point $X\in \mathbb{R}^n$ and $K$ is measured by the usual infimum formula
$$d(X, K) := \inf  \{ d(X,Y) \ | \ Y\in K \}.$$
The equidistant set of $K$ and $L\subset \mathbb{R}^n$ is defined as 
$$\{K=L\}:=\{X\in \mathbb{R}^n\ | \ d(X,K)=d(X,L)\}.$$
Since the classical conics can also be given in this way \cite{PS}, the equidistant sets are their ge\-ne\-ralizations: $K$ and $L$ are called the focal sets. Moreover, any convex polytope can be given as an equdistant set with finitely many focal points in $K$ and $L$, respectively \cite{plane}, see also \cite{space}. Therefore the equidistancy is the generalization of convexity. In a similar way we can speak about equidistant functions \cite{equidfunc} by requiring its epigraph to be an equidistant body:
$$\{K\leq L\}:=\{X\in \mathbb{R}^n\ | \ d(X,K)\leq d(X,L)\}.$$
In case of singletons we are going to use the applicable shortcut notations 
$$\{X\leq L\}:=\{\{X\}\leq L\}, \ \{X\leq Y\}:=\{\{X\}\leq \{Y\}\} \ \textrm{etc.} \ \ (X, Y\in \mathbb{R}^n).$$
The investigation of equidistant sets with finitely many focal points is motivated by a continuity theorem \cite{PS}: If $K$ and $L$ are disjoint compact subsets in the space, $K_n\to K$ and $L_n\to L$ are convergent sequences of non-empty compact subsets with respect to the Hausdorff metric, then $\{K_n=L_n\}$ is a convergent sequence tending to $\{K=L\}$ with respect to the Hausdorff metric in any bounded region of the space. For the Hausdorff distance between compact sets in $\mathbb{R}^n$ see e.g. \cite{Lay}. 

The points of an equidistant set are difficult to determine in general because there are no simple formulas to compute the distance between a point and a set. The continuity theorem allows us to simplify the general problem by using the approximation $\{K_n=L_n\} \approx \{K=L\}$, where $K_n\subset K$ and $L_n\subset L$ are finite subsets. In case of finite focal sets in 2D, the equidistant points can be characterized in terms of computable constants and parametrization \cite{planeequid} (the paper contains a MAPLE implementation as well, as an alternative of the error estimation process for quasi-equidistant points suggested by \cite{PS} for the computer simulation). For some general investigations of equidistant sets in metric spaces we can refer to Loveland's and Wilker's fundamental works  \cite{Loveland} and \cite{Wilker}. 

\section{Convex components, graph representations and connectedness}

\begin{lemma}
\label{lem:00}
Let $K$ and $L$ be non-empty compact subsets in the Euclidean coordinate space $\mathbb{R}^n$. The equidistant body
$$\{K\leq L\}:=\{X\in \mathbb{R}^n\ | \ d(X,K)\leq d(X,L)\}$$
can be expressed as the union
$$\{K\leq L\}=\bigcup_{X\in K} \{ X \leq L\}.$$
Especially
$$\{K_1 \cup K_2 \leq L\}=\{K_1\leq L\} \cup \{K_2 \leq L\}.$$
\end{lemma}

\begin{proof}
If $Z\in \{K\leq L\}$ then $d(Z, K)\leq d(Z, L)$ and, by the compactness (especially, the closedness) there exists a point $X\in K$, where the minimal distance is attained at:
$$d(Z,X)=d(Z, K)\leq d(Z,L) \ \Rightarrow \ Z\in \{ X \leq L\}.$$
Conversely, for any $X\in K$, $d(Z, K)\leq d(Z, X),$ i.e. if $Z\in \{ X \leq L\}$, then 
$$d(Z, K)\leq d(Z, X)\leq d(Z,L)$$
and $Z\in \{K\leq L\}$ as was to be proved. 
\end{proof}

The previous result motivates us to formulate some simple observations about the equidistant bodies of the form $\{ X \leq L\}$.

\begin{lemma}
\label{lem:03}
For any $X\in \mathbb{R}^n$ the set $\{X\leq L\}$ is convex. If $X$ is not an accumulation point of $L$ then $\{X\leq L\}$ is of dimension $n$ . 
\end{lemma}

\begin{proof} The convexity follows from the halfspace intersection formula
\begin{equation}
\label{halfspaceinter}
\{ X \leq L\}=\bigcap_{Y\in L} \{ X \leq Y \},
\end{equation}
where $\{ X \leq Y\}$ is the closed halfspace bounded by the perpendicular bisector of the segment $XY$  containing $X$. On the other hand, if $X$ is not an accumulation point of $L$, then it is an outer point or an isolated point. In case of an outer point $d(X,L)>0$ and a continuity argument shows that $X$ is an interior point of $\{X\leq L\}$. Otherwise (in case of an isolated point of $L$) 
$$d(X,Z)=d(Z,L)$$
for any element $Z$ in a sufficiently small open neighbourhood of $X$. Therefore $X$ is an interior point of $\{X\leq L\}$.
\end{proof}

\begin{remark} {\emph{Note that the converse does not hold in general: if $X=(0,0)$, $L=\{(0, 1/n)\ | \ n\in \mathbb{N}\}$ then 
$$\{X\leq L\} =\{(x, y)\ | \ x\leq 0\}$$
is of dimension $2$ but $X$ is an accumulation point of $L$. It can be easily seen that 
\begin{itemize}
\item if $X$ is an inner point of $L$ then $\{X\leq L\}=\{X\}$,
\item if $X$ is an outer point or an isolated point of $L$ then $X$ is an inner point of $\{X\leq L\}$, 
\item if $X$ is an accumulation point of $L$ then $-X+\{X\leq L\}$ is a subset in the regular normal cone \cite{RW} of the topological closure of $L$ at the point $X$ containing proximal normals of the form $Z-X$ $(Z\in \{X\leq L\})$.
\end{itemize}}} 
\end{remark} 

By Lemma \ref{lem:00} any equidistant body can be expressed as the union of convex subsets determined by its focal points. Following the steps in the proof we also have that  
$$\{K < L\}:=\{X\in \mathbb{R}^n\ | \ d(X,K)<d(X,L)\}$$
can be expressed as the union
$$\{K < L\}=\bigcup_{X\in K} \{ X < L\}$$
and, consequently,
$$\{K\leq L\}=\bigcup_{X\in K} \{ X \leq L\}=\bigcup_{X\in K} \overline{\{ L < X \}}=$$
$$\overline{\bigcap_{X\in K} \{ L < X \}}=\overline{\bigcap_{X\in K} \bigcup_{Y\in L}\{ Y < X \}}=\bigcup_{X\in K} \bigcap_{Y\in L}\{ X \leq Y \}.$$
On the other hand
$$\overline{\{K\leq L\}}=\{L< K\}=\bigcup_{Y\in L} \{ Y < K\}=\bigcup_{Y\in L} \overline{\{ K \leq Y\}}=\overline{\bigcap_{Y\in L} \{ K\leq Y\}} \ \Rightarrow \ $$
$$\{K\leq L\}=\bigcap_{Y\in L} \{ K\leq Y\}=\bigcap_{Y\in L} \bigcup_{X\in K}\{ X\leq Y\}.$$
Especially,
\begin{equation}
\label{intersection}
\{K\leq L_1 \cup L_2\}=\{K\leq L_1\} \cap \{K \leq L_2\}.
\end{equation}
\begin{lemma}
\label{lem:01}
Let $K$ and $L$ be non-empty compact subsets in the Euclidean coordinate space $\mathbb{R}^n$. The equidistant body
$$\{K\leq L\}:=\{X\in \mathbb{R}^n\ | \ d(X,K)\leq d(X,L)\}$$
is bounded if and only if $K$ is in the interior of the convex hull of $L$.
\end{lemma}

\begin{proof}  Suppose that the equidistant body is bounded and $X_1 \in K$ such that $X_1$ is not in the interior of the convex hull of $L$. Consider the closest point $Y_1$ of $\textrm{conv\ } L$ to $X_1$. It is uniquely determined. Taking the supporting hyperplane $H_1$ to $\textrm{conv\ } L$ at $Y_1$ such that 
\begin{itemize}
\item it is perpendicular to the segment $X_1 Y_1$ in case of $X_1\neq Y_1$ (see Figure 1), 
\item it is an arbitrary supporting hyperplane in case of $X_1=Y_1$,
\end{itemize}
the halfspace containing the convex hull of $L$ is called positive. Its complement is called negative. Choosing a point $Q$ of the ray emanating from $X_1$ in the negative open half space along the orthogonal direction to $H_1$, it can be easily seen that $Y_1$, $X_1$ and $Q$ is collinear, i.e.
$$d(Q,K)\leq d(Q, X_1) \leq d(Q, Y_1)=\inf_{Y\in \ \textrm{conv\ } L} d(Q, Y)\leq \inf_{Y\in L} d(Q, Y)=d(Q, L)$$
because of $L\subset \ \textrm{conv\ } L$. Therefore $Q\in \{K\leq L\}$ but $Q$ can tend to the infinity. It is a contradiction. 
\begin{figure}
\centering
\includegraphics[scale=0.25]{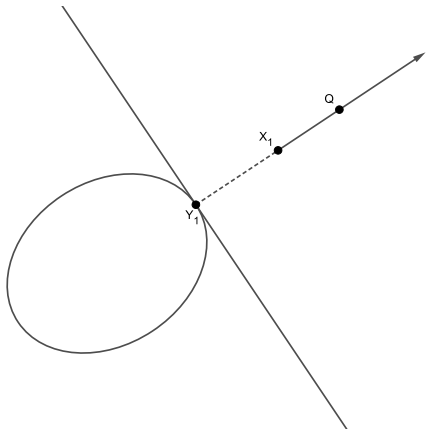}
\caption{The proof of Lemma \ref{lem:01}.}
\end{figure}
To prove the converse statement suppose that $K$ is contained in the interior of the convex hull of $L$. By Lemma \ref{lem:00}, for any $\displaystyle{Q\in \{K\leq L\}}$ we have that $Q\in \{X\leq L\}$ for some point $X$ in $K$. This means that
$$d(Q,X)\leq d(Q, Y) \ \ (Y\in L).$$
Taking the square of both sides we have that
$$\langle Q, Y-X \rangle \leq \frac{1}{2} |Y|^2-\frac{1}{2}|X|^2 \leq c,$$
where the upper bound $c>0$ can be choosen independently of $Q$ due to the boundedness of $K$ and $L$. Therefore $Q/c$ is an element in the polar body $\displaystyle{\left(-X+\textrm{conv\ }L\right)^*}$. The translated set $\displaystyle{-X+\textrm{conv\ }L}$ contains the origin in its interior because $X\in K$ and $K$ is in the interior of the convex hull of $L$. Taking a sufficiently small radius $r_X>0$ we have that 
$$r_X D \subset -X+\textrm{conv\ }L \ \Rightarrow \ Q/c \in \left(-X+\textrm{conv\ }L\right)^* \subset \left(r_X D \right)^*=D/r_X \ \Rightarrow \ Q \in c D/r_X,$$
where $D$ is the closed unit ball around the origin in the space. Using a compactness argument, it follows that
$$r:=\inf \sup_{X\in K} \{ r_X \ | \ r_X D\subset -X+\textrm{conv\ } L\}=\inf \sup_{X\in K} \{ r_X \ | \ X+r_X D\subset \textrm{conv\ } L\}>0,$$
i.e. $\displaystyle{Q \in c D/r}$ for any $Q\in \{K\leq L\}$. 
\end{proof}

\subsection{The graph representation of equidistant bodies with finitely many focal points}

In what follows some basic properties (connectedness) of an equidistant body will be investigated by the pairwise comparison of the convex components in the union
$$\{K\leq L\}=\bigcup_{i=1}^p \{X_i \leq L\},$$
where $K=\{X_1, \ldots, X_p\}$, $L=\{Y_1, \ldots, Y_q\}$ are finite sets. It is motivated by the difficulties of the description of the geometric relationships among the focal points. The pairwise comparison of simple configurations   seems to be a more effective way according to the algorithmic methods as well (see e.g. the finite version of Helly's theorem). 

\begin{definition}
The vertices of the graph representation $G$ of the equidistant body $\{K\leq L\}$ are the elements of $K$ and there is an edge between $X_i$ and $X_j$ $(i\neq j)$ if and only if
$$\{X_i \leq L\}\cap \{X_j \leq L\}\neq \emptyset.$$
The weight of the edge $X_iX_j$ is
$$w_{ij}=\dim \{X_i \leq L\}\cap \{X_j \leq L\}.$$
\end{definition}

Since
$$\{X_i \leq L\}\stackrel{(\ref{intersection})}{=}\bigcap_{k=1}^q \{ X_i \leq Y_k \},$$
the existence of the edge between $X_i$ and $X_j$ can be checked algorithmically by the finite version of Helly's  theorem. 
 
\begin{theorem}
\label{connectedness} The equidistant body $\{K\leq L\}$ is connected if and only if its graph representation is connected.
\end{theorem}

\begin{proof}
Suppose that $G=A\cup B$, where $A$ and $B$ are disjoint subsets of the vertices such that there is no edge with endpoints in $A$ and $B$, respectively. Taking the sets
$$M=\bigcup_{X_i\in A} \{X_i \leq L\} \ \textrm{and}\ N=\bigcup_{X_i\in B} \{X_i \leq L\}$$
we have that $M$ and $N$ are closed disjoint subsets of the equidistant body and $\{K\leq L\}=M \cup N$ because $G=A\cup B$. If the equidistant body is connected then one of the sets, say $M$ must be empty. So is $A$, i.e. the graph representation is connected. Conversely, if $G$ is connected then we have a sequence of edges from $X_i$ to $X_j$ for any pair of indices $i\neq j$: $X_i - X_{k_1} - \ldots - X_{k_m} - X_j$. The continuous path connecting $P\in \{X_i \leq L\}$ with $Q\in \{X_j\leq L\}$ can be constructed as follows: the first step is to join $P$ with $X_i$ (they are in the same convex component), the second step is to join $X_i$ with a point
\begin{equation}
\label{chain}
X_{ik_1}\in \{X_i \leq L\}\cap \{X_{k_1} \leq L\}
\end{equation}
and, finally, we can join $X_{i k_1}$ with $X_{k_1}$ because they are in the same convex component. The polygonal chain constructed by the succesive application of these steps joins $P$ with $Q$. Therefore the body is arcwise connected, i.e. it is connected. 
\end{proof}

According to the argument in the proof of the previous theorem, the connectedness and the arcwise connectedness are equivalent for an equidistant body. 

\begin{corollary}
The equidistant body $\{K\leq L\}$ is arcwise connected if and only if its graph representation is connected.
\end{corollary}

\begin{proof}
If the  equidistant body is arcwise connected then it is connected. So is its graph representation. Conversely, if $G$ is connected then we can follow the process in the proof of the previous theorem to construct a polygonal chain between any pair of points $P$ and $Q$ in $\{K\leq L\}$. 
\end{proof}

\begin{corollary} 
\label{discon} 
A disconnected equidistant body $\{K\leq L\}$ is the disjoint union of equidistant bodies of type $\{K_i\leq L\}$, where $K=K_1 \cup K_2 \cup \ldots$ is the disjoint union of subsets in $K$ corresponding to the connected components of the graph representation. 
\end{corollary}

Some disconnected cases are illustrated in Figure 2: the equidistant body is disconnected (left), the equidistant body is connected but its interior is not (right). Figure 3 shows the case of a not connected complement.

\begin{figure}
\centering
\includegraphics[scale=0.4]{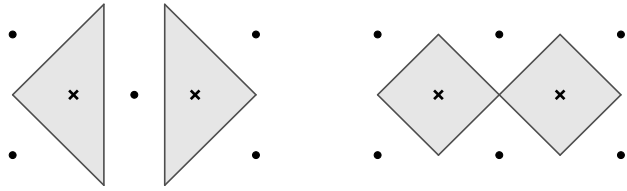}
\caption{Some disconnected cases.}
\end{figure}

\begin{figure}
\centering
\includegraphics[scale=0.4]{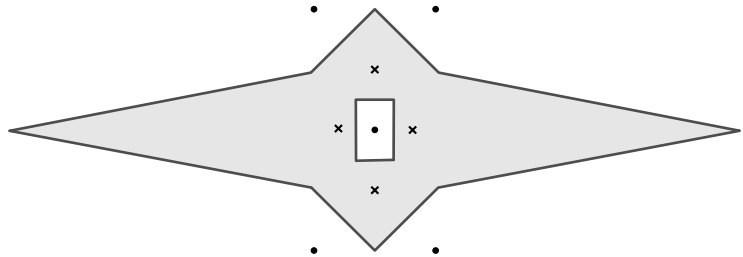}
\caption{Disconnected components of the complement of an equidistant body.}
\end{figure}

\begin{definition}
The graph representation $G$ of the equidistant body $\{K\leq L\}$ is disconnected with respect to the weight $w$ if $G=A\cup B$, where $A$ and $B$ are non-empty disjoint subsets of the vertices such that there are no edges of weight greater  or equal than $w$ with endpoints in $A$ and $B$, respectively. Otherwise $G$ is connected with respect to the weight $w$.
\end{definition}

\begin{theorem} The interior of the equidistant body $\{K\leq L\}$ is connected if and only if its graph representation is connected with respect to the weight $n-1$.
\end{theorem}

\begin{proof}
Suppose that the graph representation is connected with respect to the weight $n-1$ and let us modify the proof of Theorem \ref{connectedness} by changing the intermediate points $X_{ik_1}, X_{k_1 k_2}, \ldots, X_{k_m j}$ such that they belong to the interior of the intersection of the corresponding convex components (an edge of maximal weight) or they are in the relative interior of the adjacent faces of dimension $n-1$. The modification gives a continuous path (polygonal chain) from $X_i$ to $X_j$ ($i\neq j$) in the interior of the equidistant body. Since the arcwise connectedness implies the connectedness we are done. Conversely, suppose that the interior of the equidistant body $\{K\leq L\}$ is connected, i.e. it is arcwise connected because the connectedness and the arcwise connectedness are equivalent in case of open sets. To prove the connectedness of the graph representation with respect to the weight $n-1$ we are going to construct a sequence of edges of weight at least $n-1$ between $X_i$ and $X_j$ ($i\neq j$). If $X_j \in \{X_i \leq L\}$ then we are done because $\dim \{X_i \leq L\}\cap \{X_j \leq L\}=n$. Indeed, since a finite set has no accumulation points, Lemma \ref{lem:03} implies that $X_j$ is in the interior of $\{X_j \leq L\}$. Especially, each convex component is of dimension $n$. Finally, if a convex set of dimension $n$ intersects the interior  of another one then the intersection is of dimension $n$. Otherwise consider a continuous path $\widehat{X_i X_j}$ from $X_i$ to $X_j$ in the interior of the equidistant body and choose a common point $Z$ of $\widehat{X_i X_j}$ with the boundary of $\{X_i \leq L\}$. Since $Z$ is an interior point, we can suppose - without loss of generality - that $Z$ is lying on the face $F_{n-1}^i$ of dimension $n-1$ of the convex component $\{X_i \leq L\}$. Let $\varepsilon>0$ be small enough and $B_Z(\varepsilon)\subset \textrm{int\ } \{K\leq L\}$, where $B_Z(\varepsilon)$ is the open ball around $Z$ with radius $\varepsilon$. Then $B_Z(\varepsilon)\cap F_{n-1}^i$ must be covered by the finite collection $\{ \{X_k \leq L\} \ | \ k\neq i \}.$ Condition $\dim B_Z(\varepsilon)\cap F_{n-1}^i=n-1$ implies that there must be at least one intersection $\{X_i \leq L\}\cap \{X_{k_1} \leq L\}$ of dimension at least $n-1$. Therefore $w_{i k_1}$ (the weight of the edge $X_i X_{k_1}$) is at least $n-1$. Repeating the algorithm along the arc $\widehat{X_{k_1} X_j}$  we are done in finitely many steps.
\end{proof}

\subsection{Equidistant polytopes}

In what follows we are going to define the notion of equidistant polytopes. Some natural requirements are the connectedness (see Corollary \ref{discon}), the boundedness (see Lemma \ref{lem:01}) and the finiteness of the focal sets.

\begin{definition}
\label{def:01} Let $K$ and $L\subset \mathbb{R}^n$ be non-empty disjoint, finite sets and suppose that $K$ is a subset in the interior of the convex hull of $L$. The equidistant body $\{K\leq L\}$ is called an equidistant polytope if both its interior and its complement are connected. The equidistant polytope is of type $(q, p)$, where $|L|=q$ and $|K|=p$. 
\end{definition}

\begin{corollary}
\label{connectedboundary}
The boundary $\{K=L\}$ of an equidistant polytope is connected. 
\end{corollary}

The equidistant polytopes of type $(q,1)$ are convex polytopes. It is a direct consequence of the halfspace intersection formula (\ref{halfspaceinter}) and Lemma \ref{lem:01}: an equidistant polytope of type $(q, 1)$ is a non-empty, compact intersection of finitely many closed halfspaces. Since any convex polytope can be given as an equidistant polytope of type $(q,1)$, the equidistancy is the generalization of the convexity, see \cite{plane} and \cite{space}. In a similar way we can speak about equidistant functions \cite{equidfunc} by requiring its epigraph to be an equidistant body. Let $\{K\leq L\}$ be an equidistant polytope; since $K$ must be in the interior of the convex hull of $L$, the set of the outer focal points must contain at least $n+1$ points. 

\begin{lemma}
\label{cor:01}
If $K$ and $L\subset \mathbb{R}^n$ are non-empty disjoint finite sets, $|L|=n+1$ and $K$ is a subset in the interior of the convex hull of $L$, then the equidistant body $\{K\leq L\}$ is a star-shaped set.  
\end{lemma}

\begin{proof} The idea is to provide the existence of a closed ball strictly separating $K$ and $L$ in the sense that the focal points of $K$ are inside but the focal points of $L$ are outside\footnote{The combinatorial  criteria of the existence of such a separating ball can be formulated as a Kirchberger-type theorem \cite{Lay}: if for any subset $T\subset K \cup L$ containing $n+3$ points there is a ball strictly separating $T\cap K$ and $T\cap L$ then there is a ball strictly separating $K$ and $L$. }. Then the center $X_*$ of the ball satisfies the inequality
$$\max \{d(X_*, X_1), \ldots, d(X_*, X_p)\} < d(X_*,L)=\min \{d(X_*, Y_1), \ldots, d(X_*, Y_q)\},$$
where $X_1, \ldots, X_p$ are the points in $K$ and $Y_1, \ldots, Y_q$ are the points in $L$. Therefore, by a continuity argument, $X_*$ is an interior point of every convex component $\{X_i\leq L\}$, where $i=1, \ldots, p$. By Lemma \ref{lem:00}, this means that the equidistant body $\{K\leq L\}$ is star-shaped with respect to any point in an open ball around $X_*$ with a sufficiently small radius. If $q=n+1$ then $L$ is a simplex and we can easily construct a strictly separating ball by a slight decreasing of the radius of its circumscribed sphere.
\end{proof}

\begin{corollary}
If $K$ and $L\subset \mathbb{R}^n$ are non-empty disjoint finite sets, $|L|=n+1$ and $K$ is a subset in the interior of the convex hull of $L$, then the equidistant body $\{K\leq L\}$ is an equidistant polytope.   
\end{corollary}

\subsection{Voronoi decomposition and its correspondence to the equidistant bodies} One of the most important application of equidistant bodies of the form $\{X\leq L\}$ is the Voronoi decomposition of the space. Let $K:=\{X_1, \ldots, X_m\} \subset \mathbb{R}^n$ be a finite set containing different elements and consider the equidistant bodies
$$V(X_i, K):=\{X_i \leq K\setminus \{X_i\}\} \quad (i=1, \ldots, m).$$
It is clear that they are $n$-dimensional convex subsets (Voronoi cells) such that 
$$\mathbb{R}^n=\bigcup_{i=1}^m V(X_i, K)
$$
and
$$\ \textrm{int} \  V(X_i, K) \cap \textrm{int}\  V(X_j, K)=\emptyset \quad (i\neq j=1, \ldots, n).$$
The set $V(X_i, K)$ contains the points in the space, where the distance $d(X,K)$ is attained at $X_i$, i.e. $X_i$ is the closest point of $K$ to any $X\in V(X_i, K)$. The collection of the cells $V(X_1, K)$, $\ldots$, $V(X_m, K)$ is called the Voronoi decomposition of the space with respect to the set $K$.
\begin{corollary} If $K$ and $L$ are non-empty finite, disjoint subsets containing different elements then  
$$\{K\leq L\}=\bigcup_{X_i\in K} V(X_i, K \cup L).$$
\end{corollary}

\section{Equidistant polygons in the plane: the hypergraph representation and the maximal number of the vertices}

The following investigations are motivated by the discrete version of the problem posed in \cite{PS}: \emph{characterize all closed sets of the plane that can be realized as the equidistant set of two connected disjoint closed sets.}

\subsection{The hypergraph representation of equidistant polygons} 

Let $K$ and $L\subset \mathbb{R}^2$ be non-empty disjoint finite subsets in the plane such that $K$ is contained in the interior of the convex hull of $L$. In what follows we are going to estimate the maximal number of the edges (vertices) of an equidistant polygon in terms of the number of elements in the focal sets. Using a continuity argument it can be easily seen that the decreasing of the number of the vertices is impossible by a slight modification of the position\footnote{ Since the convex components are non-empty compact intersections of finitely many half-spaces, they depend continuously on the position of the focal points in any bounded region of the space. So does their finite union. Therefore vertices (breakages along the boundary) cannot be straightened by a slight modification of the position of the focal points but straight line segments can be broken.} of the focal points (the increasing of the number of the vertices is possible). This means that the regularity conditions
\begin{itemize}
\item [(C1)] there are no collinear triplets among the points of $K\cup L$,
\item [(C2)] there are no concircular quadruples among the points of $K\cup L$
\end{itemize}
can be supposed without loss of generality. The regularity conditions do not imply in general that we have an equidistant polygon as Figure 3 shows by a slight modification of the focal sets. In the special case of 2D, the connectedness of the interior of an equidistant polygon provides that there are no self-intersections of its boundary and the connectedness of its complement provides that there are no holes in its interior. Using the boundedness criterion ($K\subset \textrm{int\ conv\ } L$) and the finiteness of the focal sets, Corollary \ref{connectedboundary} implies that the boundary of an equidistant polygon in the plane is a simple closed polygonal chain (the edges belong to the perpendicular bisectors of the elements in the focal sets $K$ and $L$, respectively). Therefore an equidistant polygon in the plane is a Jordan polygon (cf. the Jordan curve theorem). 

\begin{definition} The $3$-uniform hypergraph representation of an equidistant polygon satisfying $(C1)$ and 
$(C2)$ consists of the vertices $K\cup L$ and the edges are the triplets of the focal points provided that the circle determined by them does not contain any focal point in its interior. Edges of types $\{X_{i_1}, X_{i_2}, X_{i_3}\}$ or $\{Y_{j_1}, Y_{j_2}, Y_{j_3}\}$ are called monochromatic. The colored edges are of types $\{X_{i_1}, Y_{j_1}, X_{i_2}\}$, or $\{Y_{j_1}, X_{i_1}, Y_{j_2}\}$, respectively. The weight of a colored edge is  the angle $\delta_{i_1 i_2}^{j_1}=\angle X_{i_1} Y_{j_1} X_{i_2}$, or $\omega_{j_1 j_2}^{i_1}=\angle Y_{j_1} X_{i_1} Y_{j_2}$.
\end{definition}

To prevent the inclusion of focal points in the interior of the circle determined by a colored edge it is sufficient and necessary for the weight to satisfy 
$$\delta_{i_1 i_2}^{j_1}=\max\{\angle X_{i_1} Z X_{i_2} \ | \ Z\in K\cup L \ \textrm{and} \ Z\in H_K^+ (i_1, j_1, i_2)\}\leq $$
$$\pi - \max\{\angle X_{i_1} Z X_{i_2} \ | \ Z\in K\cup L \ \textrm{and} \ Z\in H_K^- (i_1, j_1, i_2)\},$$
where the open halfplane $H_K^+ (i_1, j_1, i_2)$ is bounded by the line $X_{i_1} X_{i_2}$ and it contains the point $Y_{j_1}$,  $H_K^- (i_1, j_1, i_2)$ is the opposite open halfplane, or
$$\omega_{j_1 j_2}^{i_1}=\max\{\angle Y_{j_1} Z Y_{j_2} \ | \ Z\in K\cup L \ \textrm{and} \ Z\in H_L^+ (j_1, i_1, j_2)\}\leq$$
$$ \pi - \max\{\angle  Y_{j_1} Z Y_{j_2} \ | \ Z\in K\cup L \ \textrm{and} \ Z\in H_L^- (j_1, i_1, j_2)\},$$
where the open halfplane $H_L^+(j_1, i_1, j_2)$ is bounded by the line $Y_{j_1} Y_{j_2}$ and it contains the point $X_{i_1}$, $H_L^- (j_1, i_1, j_2)$ is the opposite open halfplane.

\begin{corollary}
\label{vertices} Consider the hypergraph representation of an equidistant polygon satisfying $(C1)$ and $(C2)$. The centers of the circles determined by monochromatic edges of type $\{X_{i_1}, X_{i_2}, X_{i_3}\}$ and $\{Y_{j_1}, Y_{j_2}, Y_{j_3}\}$ are interior and exterior points of the equidistant polygon, respectively. The centers of the circles determined by colored edges are the vertices of the equidistant polygon. 
\end{corollary}

\begin{proof}
If $\{X_{i_1}, X_{i_2}, X_{i_3}\}$ is a monochromatic edge then there are no focal points in the interior of the circle determined by $X_{i_1}$, $X_{i_2}$ and $X_{i_3}$. The existence of such a circle is due to $(C1)$. Let $r_{i_1 i_2 i_3}$ and $X_{i_1 i_2 i_3}$ be the radius and the center of the circle, respectively. We have that 
$$d(X_{i_1 i_2 i_3}, K)= r_{i_1 i_2 i_3} < \min \{d(X_{i_1 i_2 i_3}, Y_1), \ldots, d(X_{i_1 i_2 i_3}, Y_q)\}=d(X_{i_1 i_2 i_3}, L),$$
where the strict inequality is due to $(C2)$. Therefore $X_{i_1 i_2 i_3}$ is in the interior of $\{K\leq L\}$. The argument is similar in case of a monochromatic edge $\{Y_{j_1}, Y_{j_2}, Y_{j_3}\}$. Taking a colored edge $\{X_{i_1}, Y_{j_1}, X_{i_2}\}$ let $r_{i_1 j_1 i_2}$ and $V_{i_1 j_1 i_2}$ be the radius and the center of the circle determined by the points of the triplet, respectively. The existence of such a circle is due to $(C1)$. We have that  
$$d(V_{i_1 j_1 i_2}, K)=r_{i_1 j_1 i_2}=d(V_{i_1 j_1 i_2}, Y_{j_1})=d(V_{i_1 j_1 i_2}, L).$$
This means that $V_{i_1 j_1 i_2}$ is an equidistant point of $K$ and $L$. The perpendicular bisectors of the chords $X_{i_1}Y_{j_1}$ and $X_{i_2}Y_{j_1}$ intersect each other at $V_{i_1 j_1 i_2}$. According to $(C2)$ there are no equidistant points in a sufficiently small open neighbourhood of $V_{i_1 j_1 i_2}$ except the points of the perpendicular bisectors. They determine a concave angle because  $X_{i_1}$ and $X_{i_2}$ are automatically in the interior of $\{K\leq L\}$. The argument is similar in case of a colored edge $\{Y_{j_1}, X_{i_1}, Y_{j_2}\}$. 
\end{proof}

The colored edge $\{X_{i_1}, Y_{j_1}, X_{i_2}\}$ represents a single inner change in the sense that the vertex (the center of the circle determined by the elements of the triplet) is due to the change of the inner focal points (the outer focal point is the same). The circle is passing through exactly two of the inner and exactly one of the outer focal points (concave angles). The colored edge $\{Y_{j_1}, X_{i_1}, Y_{j_2}\}$ represents a single outer change in the sense that the vertex (the center of the circle determined by the elements of the triplet) is due to the change of the outer focal points (the inner focal point is the same). The circle is passing through exactly two of the outer and exactly one of the inner focal points (convex angles). Condition $(C2)$ does not allow "double changes" in the sense that the vertex is due to the simultaneous change of the outer and the inner focal points. Such a kind of change will appear among the cases of the special arrangements of the focal points: Figure 8 shows a double change at $V_1$, single outer changes at $V_2$ and $V_4$, a single inner change at $V_3$. 

\begin{figure}
\centering
\includegraphics[scale=0.35]{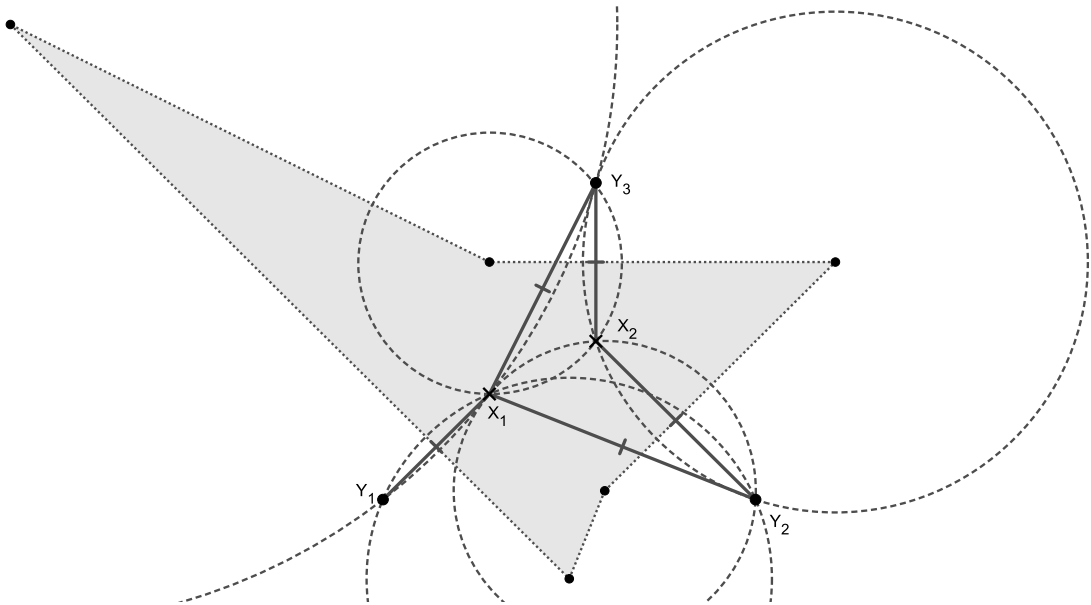}
\caption{The bigraph of the pentagon: the proof of Lemma \ref{vertexmax}.}
\end{figure}

\begin{lemma}
\label{vertexmax} An equidistant polygon of type $(q,p)$ has at most $p+q$ vertices.
\end{lemma}

\begin{proof} Using a continuity argument it follows that the decreasing of the number of the vertices is impossible by a slight modification of the position of the focal points (the increasing of the number of the vertices is possible, see footnote 2). Therefore we can suppose that $(C1)$ and $(C2)$ are satisfied. Taking the hypergraph representation suppose that it is minimal in the sense that only the focal points belonging to colored edges are considered. This means that we have a finite chain of circles such that the centers form the vertices of the equidistant polygon in a given direction. The adjacent vertices correspond to adjacent circles having a common chord with endpoints $X_{i_m}\in K$ and $Y_{j_m}\in L$. Let us choose a starting vertex/circle. The following algorithm generates a bigraph (see Figure 4) with edges
\begin{itemize}
\item [(i)] $e_1:=X_{i_1} Y_{j_1}$,
\item [(ii)] if $e_m:=X_{i_m} Y_{j_m}$ and $(X_{i_m}, Y_{j_m}, Z)$ determines the adjacent circle with respect to the given direction then
$$e_{m+1}=\left\{
\begin{array}{rl}
X_{i_m} Z &\ \textrm{if } \ Z\in L\\
Y_{j_m} Z &\ \textrm{if } \ Z\in K.
\end{array}
\right.
$$
\end{itemize}
There is a one-to-one correspondence between the edges and the circles. Therefore the number of the edges equals to the number of the circles (the number of the vertices of the equidistant polygon). On the other hand, exactly one new element in $K\cup L$ appears in each step. This means that the number of the circles (the number of the vertices of the equidistant polygon) is less or equal than $p+q$ as was to be proved.
\end{proof}

\begin{remark}
{\emph{We can improve the estimation in case of $p=1$ as follows: for the number of the vertices 
$$|(p,q)|=\left\{
\begin{array}{rl}
q&\ \textrm{if } \ p=1\\
p+q&\ \textrm{otherwise}
\end{array}
\right.
$$
provided that the hypergraph representation is minimal. Indeed, each point in the minimal representation $K\cup L$ appears in the matching process (i) and (ii) and each pair in the matching correspond to a consecutive circle. The edges $e_1$, $\ldots$, $e_m$, $\ldots$ of the bigraph are orthogonal segments to the edges of the equidistant polygon.
}}
\end{remark}

\section{Equidistant polygons of type $(3,2)$ in the plane: the generic case}

Suppose that we have an equidistant polygon of type $(3,2)$ satisfying $(C1)$ and $(C2)$, $K=\{X_1, X_2\}$ and $L=\{Y_1, Y_2, Y_3\}.$ Using Lemma \ref{vertexmax} the maximal number of the vertices is $5$ and it can be attained as we shall see. Let $\omega_{12}^i$, $\omega_{23}^i$ and $\omega_{31}^i$ be the viewing angles under which the segments $Y_1Y_2$, $Y_2Y_3$ and $Y_3Y_1$ are visible from the inner focal point $X_i$ ($i=1, 2$). It is clear that
\begin{equation}
\label{anglesum}
\omega_{31}^i=360^{\circ}-\omega_{12}^i-\omega_{23}^i \quad (i=1, 2).
\end{equation}
We also introduce the viewing angle $\delta^j:=\delta_{12}^j$ under which the segment $X_1X_2$ is visible from $Y_j$  ($j=1, 2, 3$). According to condition $(C2)$ we can suppose that  
\begin{equation}
\label{pentagon}
\omega_{12}^1 > \omega_{12}^2, \ \omega_{23}^{1} < \omega_{23}^2, \ \omega_{31}^1 > \omega_{31}^2
\end{equation}
because $\omega_{23}^{1}> \omega_{23}^2$ implies that $\omega_{31}^1 < \omega_{31}^2$ and the ordering (\ref{pentagon}) follows by changing the role of $Y_1$ and $Y_2$. This means that $\{Y_1, X_1, Y_2\}$, $\{Y_2, X_2, Y_3\}$ and $\{Y_3, X_1, Y_1\}$ are colored edges in the hypergraph representation. They correspond to the convex angles of the equidistant polygon at $V_1$, $V_3$ and $V_5$ (Figure 5). What about the viewing angles $\delta^1$, $\delta^2$ and $\delta^3$? Since condition $(C1)$ is satisfied, the line $X_1X_2$ strictly separates two focal points from the third one, say $Y_3$. Using condition $(C2)$ we can suppose that $\delta^1 < \delta^2$ and, consequently, the centers of the circles (colored edges) $\{X_1, Y_2, X_2\}$ and $\{X_1, Y_3, X_2\}$ are vertices of the equidistant polygon, where concave angles appear at. Therefore we have a pentagon with exactly two concave angles at $V_2$ and $V_4$ (Figure 5). The focal points $X_1$ and $X_2$ are obviously symmetric about the line $V_2 V_4$.

\begin{figure}
\centering
\includegraphics[scale=0.3]{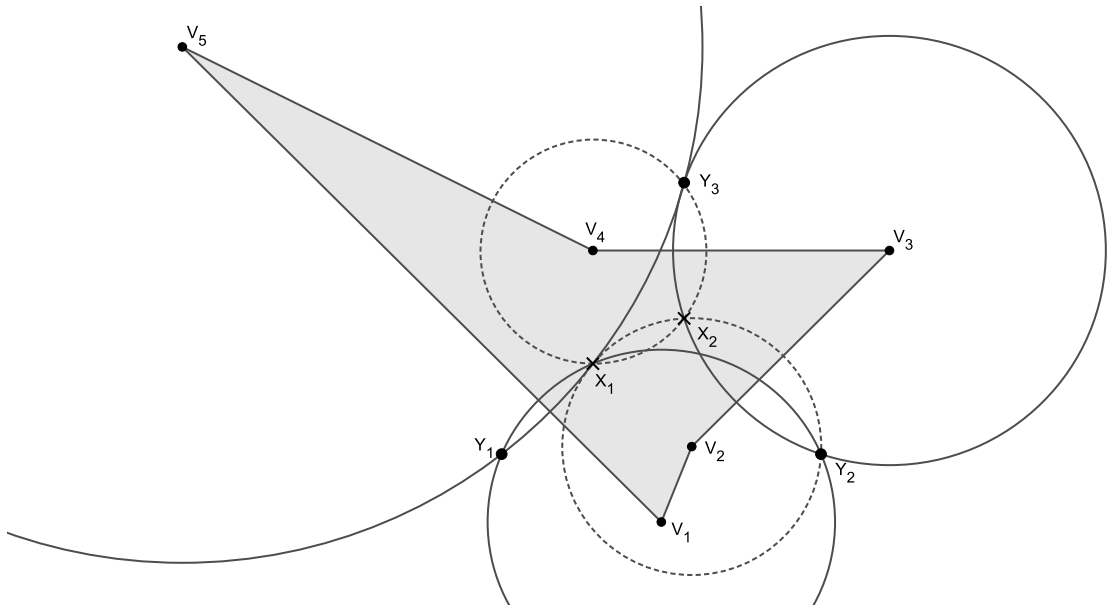}
\caption{An equidistant polygon of type $(3,2)$: the generic case.}
\end{figure}

\begin{lemma}
\label{lem:05}
Consider a simple pentagon $P$ with exactly two concave angles and let the vertices be labelled  by $A$, $B$, $C$, $D$ and $E$ in the counterclockwise direction such that the concave angles are at $B$ and $D$. If the auxiliary lines $f_B$ and $f_D$ are defined by 
$$\rho_{f_B}=\rho_{BA}\circ \rho_{BC}\circ \rho_{BD}, \ \ \rho_{f_D}=\rho_{DE}\circ \rho_{DC}\circ \rho_{DB},$$
where $\rho_{BA}$, $\rho_{BC}$, $\ldots$ denote the reflections about the lines determined by the indices, then $f_B$ and $f_D$ intersect each other on the side of the inner diagonal $BD$ containing $A$.
\end{lemma}

\begin{proof} First of all note that $f_B$ and $f_D$ are well-defined due to the three reflection theorem for concurrent lines. The theorem states that the composition of reflections about three concurrent lines is a reflection about a line passing through the common point. Figure 6 shows that the angle between the lines $f_B$ and $BD$ on the side of  the inner diagonal $BD$ containing $A$ is just $\angle B-\pi$, where $\angle B$ is the concave angle of the polygon at $B$. In a similar way, $\angle D-\pi$ is the angle enclosed by $f_D$ and $DB$ on the side of $DB$ containing $A$. Since
$$\angle A+\angle B+\angle C+\angle D+\angle E=3\pi,$$
it follows that
$$(\angle B-\pi)+(\angle D-\pi) = \pi-(\angle A+\angle C+\angle E)<\pi$$
and the intersection point of $f_B$ and $f_D$ exists on the side of the inner diagonal $BD$ containing $A$.
\end{proof}

\begin{figure}
\centering
\includegraphics[scale=0.4]{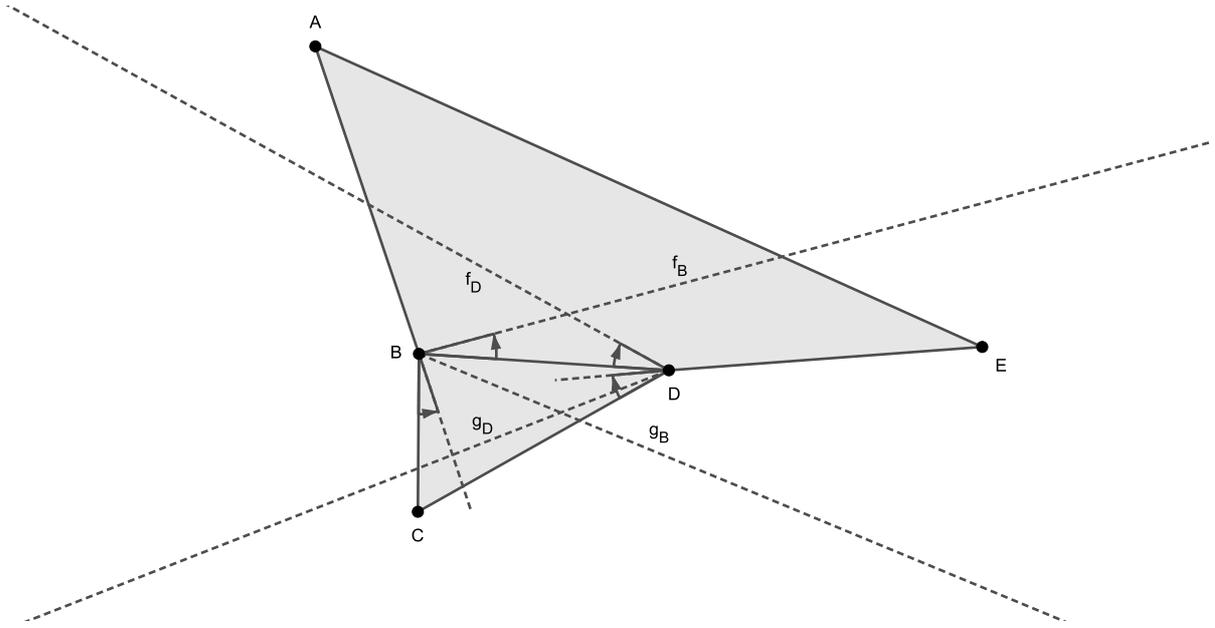}
\caption{The proof of Lemma \ref{lem:05}: pseudo inner focal points.}
\end{figure}

\begin{corollary}
\label{cor:05}
Consider a simple pentagon $P$  with exactly two concave angles and let the vertices be labelled  by $A$, $B$, $C$, $D$ and $E$ in the counterclockwise direction such that the concave angles are at $B$ and $D$. If the auxiliary lines $g_B$ and $g_D$ are defined by 
$$\rho_{g_B}=\rho_{BC}\circ \rho_{BA}\circ \rho_{BD},\ \ \rho_{g_D}=\rho_{DC}\circ \rho_{DE}\circ \rho_{DB},$$
where $\rho_{BC}$, $\rho_{BA}$, $\ldots$ denote the reflections about the lines determined by the indices, then $g_B$ and $g_E$ intersect each other on the side of the inner diagonal $BD$ containing $C$.
\end{corollary}

\begin{proof}
Note that $f_B$ and $g_B$ (or $f_D$ and $g_D$) are symmetric about the line $BD$ because (for example) for any $G\in g_B$
$$\rho_{BD}(G)=\rho_{BD}\circ \rho_{g_B} (G)=\rho_{BD}\circ \rho_{BC}\circ \rho_{BA}\circ \rho_{BD} (G)=\rho_{f_B}^{-1}\circ  \rho_{BD} (G)=\rho_{f_B}\circ  \rho_{BD} (G),$$
i.e. $\rho_{BD}(G)\in f_B$ and vice versa.
\end{proof}

\begin{definition} Let $P$ be a simple pentagon  with exactly two concave angles such that the vertices are labelled  by $A$, $B$, $C$, $D$ and $E$ in the counterclockwise direction and the concave angles are at $B$ and $D$. The intersection points $f_B\cap f_D$ and $g_B \cap g_D$ are called the pseudo inner focal points of $P$. 
\end{definition}

\begin{figure}
\centering
\includegraphics[scale=0.4]{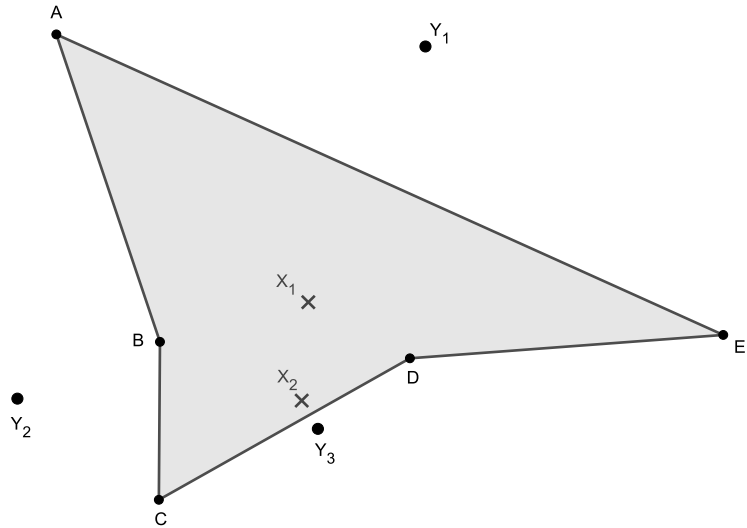}
\caption{The proof of Theorem 3.}
\end{figure}

\begin{theorem}
A simple pentagon is an equidistant polygon of type $(3,2)$ if and only if it has exactly two concave angles such that the vertices, where the concave angles appear at, are joined by an inner diagonal of the polygon and the pseudo inner focal points are in its interior. 
\end{theorem}

\begin{proof}
Suppose that $P$ is a simple pentagon satisfying the conditions of the statement. Using the notations in Figure 6
$$X_1:=f_B\cap f_D, \ X_2=g_B\cap g_D;$$
they are symmetric about the line $BD$ (see the proof of Corollary \ref{cor:05}). The outer focal points are 
$$Y_1=\rho_{AE}(X_1), \ Y_2=\rho_{AB}(X_1)=\rho_{BC}(X_2), \ Y_3=\rho_{CD}(X_2)=\rho_{DE}(X_1);$$
see Figure 7. 
\end{proof}

\section{Equidistant polygons of type $(3,2)$ in the plane: special arrangements of the focal points}

\subsection{The case of concircular points} In this case we have points in $K\cup L$, say $X_1$, $Y_1$, $X_2$ and $Y_2$ lying on the same circle. Especially, $\omega_{12}^1=\omega_{12}^2.$ Since the interior of the convex hull of $L$ contains the points in $K$, all the focal points can not be on the same circle. Without loss of generality we can suppose (by renumbering the inner focal points if necessary) that $\omega_{23}^1 < \omega^{2}_{23}$ as Figure 8 shows. Therefore $\omega_{31}^1 > \omega_{31}^2$ and there are convex angles at $V_2$ and $V_4$. Another convex angle is at the vertex $V_1$ due to the simultaneous change of the outer and the inner focal points. Since $X_1$, $Y_1$, $X_2$ and $Y_2$ are lying on the same circle, the viewing angles $\delta_1$ and $\delta_2$ are equal to each other and the secant line $X_1X_2$ strictly separates $Y_1$ and $Y_2$ from $Y_3$. This means that the center $V_3$ of the circle passing through the points $X_1$, $X_2$ and $Y_3$ is a vertex of the equidistant polygon, where a concave angle appears at. 

\begin{figure}
\centering
\includegraphics[scale=0.45]{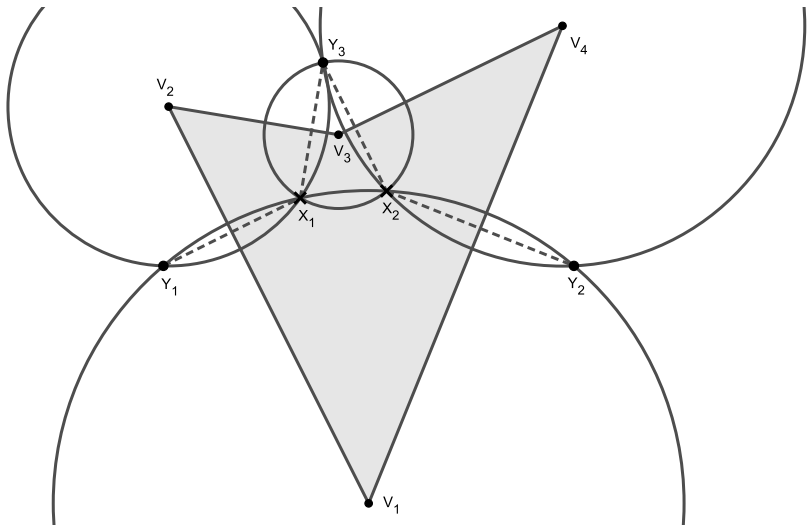}
\caption{The case of concircular points.}
\end{figure}

\begin{theorem}
\label{thm:01} A simple concave quadrangle is an equidistant polygon of type $(3,2)$.
\end{theorem}

\begin{proof}
Let the vertices of a simple concave quadrangle in the plane be labelled by $A$, $B$, $C$ and $D$ in the counterclockwise direction and suppose that the concave angle is at the vertex $C$. Let us introduce the auxiliary line $f$ passing through the vertex $C$ such that
$$\rho_f=\rho_{CD}\circ \rho_{CB}\circ \rho_{CA},$$
where $\rho_{CD}$, $\rho_{CB}$, $\ldots$ denote the reflections about the lines determined by the indices. It is well-defined due to the three reflection theorem for concurrent lines. Since $f$ passes through the vertex of the concave angle, it must contain points such that they are in the interior of the polygon together with their reflected pairs about the inner diagonal line (see Figure 9). Taking such a point $X_1$ we define
$$X_2:=\rho_{CA}(X_1).$$ 
According to the construction, $\rho_{CD}(X_1)=\rho_{CB}(X_2)=Y_3.$ Finally we complete the set of the outer focal points by $\rho_{AD}(X_1)=Y_1$ and $\rho_{AB}(X_2)=Y_2$.
\end{proof}

\begin{figure}
\centering
\includegraphics[scale=0.45]{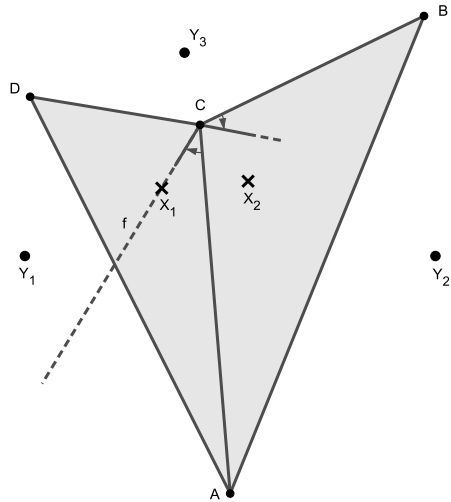}
\caption{A simple concave quadrangle as an equidistant polygon of type $(3,2)$: Theorem \ref{thm:01}.}
\end{figure}

In the proof of the previous theorem we can also consider the auxiliary line $g$ determined by 
$$\rho_g=\rho_{CB}\circ \rho_{CD}\circ \rho_{CA}$$
instead of $f$. The lines $g$ and $f$ are symmetric about the inner diagonal line $AC$ because for any $G\in g$
$$\rho_{CA}(G)=\rho_{CA}\circ \rho_g (G)=\rho_{CA}\circ \rho_{CB}\circ \rho_{CD}\circ \rho_{CA}(G)=\rho_f^{-1}\circ \rho_{CA}(G)=\rho_f\circ \rho_{CA}(G),$$
i.e. $\rho_{CA}(G)\in f$ and vice versa. Since the inner focal points must be choosen symmetrically about the inner diagonal line, the role of these lines is also symmetric in the argumentation. Indeed, $X_1 X_2$ is a common chord of the circles around $V_1$ and $V_3$ (Figure 8).

\begin{corollary} Any simple quadrangle is an equidistant polygon.
\end{corollary}

\begin{proof} Recall that convex quadrangles are equidistant polygons of type $(4,1)$; see \cite{plane}. Otherwise we  can refer to Theorem \ref{thm:01}. 
\end{proof}

\subsection{The case of collinear points} Suppose that one of the outer focal points, say $Y_1$, is collinear with $X_1$ and $X_2$. Since the inner focal points must be in the interior of the convex hull of the outer focal points, $Y_2$ and $Y_3$ must be strictly separated by the line $X_1X_2$ and, consequently, the centers of the circles $\{X_1, X_2, Y_2\}$ and $\{X_1, X_2, Y_3\}$ are vertices of the equidistant polygon, where concave angles appear at. It is the same situation as in the generic case.

\subsection{Summary} We have proved that an equidistant polygon of type $(3,2)$ in the plane belongs to one of the following classes:
\begin{itemize}
\item simple concave quadrangles (four concircular focal points, the focal sets form one-parameter families as the point $X_1$ is moving along the auxiliary line $f$),
\item simple pentagons with exactly two concave angles such that the vertices, where the concave angles appear at, are joined by an inner diagonal and the pseudo inner focal points are in the interior of the pentagon. The pseudo inner focal points are constructed by the intersections of the lines substituting the adjacent sides and the inner diagonal at the vertices, where the concave angles appear at, via the three reflection theorem for concurrent lines (the focal sets are uniquely determined).
\end{itemize}

\end{document}